\documentclass[a4paper]{amsart}

\usepackage[plainpages=false]{hyperref}
\usepackage[english]{babel}
\usepackage[latin1]{inputenc}
\usepackage{verbatim}
\usepackage[arrow, matrix, curve]{xy}
\usepackage{amsmath,amssymb,amsfonts,amsthm,amsopn}
\usepackage{nicefrac}
\usepackage{fancyhdr}
\usepackage{enumerate}
\usepackage{a4wide}
\usepackage{tikz-cd}

\newcommand{\mcG}{\mathcal{G}}
\newcommand{\norm}[1]{\lVert#1\rVert}

\DeclareMathOperator{\Gal}{Gal}

\DeclareMathOperator{\divi}{div}
\DeclareMathOperator{\sat}{sat}
\DeclareMathOperator{\tors}{tors}
\DeclareMathOperator{\End}{End}

\DeclareMathOperator{\rk}{rk}

\newtheorem{theorem}[]{Theorem}[section]
\newtheorem{conjecture}[theorem]{Conjecture}
\newtheorem{corollary}[theorem]{Corollary}
\newtheorem{proposition}[theorem]{Proposition}
\newtheorem{lemma}[theorem]{Lemma}
\newtheorem*{definition}{Definition}

\newtheorem{remark}[theorem]{Remark}

\title{Fields Generated by Finite Rank Subgroups of $\overline{\mathbb{Q}}^*$}
\author{Lukas Pottmeyer}
\address{Fakult\"at f\"ur Mathematik, Universit\"at Duisburg-Essen, 45117 Germany}
\email{lukas.pottmeyer@uni-due.de}
\date{\today}

\makeindex

\begin{document}

\begin{abstract}
Let $\Gamma$ be a finite rank subgroup of $\overline{\mathbb{Q}}^*$. We prove that the multiplicative group of the field generated by all elements in the divisible hull of $\Gamma$, is free abelian modulo this divisible hull. This proves that a necessary condition for R\'emond's generalized Lehmer conjecture is satisfied.    
\end{abstract}

\subjclass[2010]{11G50, 20K15}
\keywords{Heights, free abelian groups}
\maketitle

\section{Introduction}

We fix once and for all an algebraic closure $\overline{\mathbb{Q}}$ of the rational numbers and assume that all algebraic extensions of $\mathbb{Q}$ are contained in this closure. The absolute logarithmic Weil-height $h$ on $\overline{\mathbb{Q}}$ can be defined as follows: For $\alpha_1\in\overline{\mathbb{Q}}$ let $f(x)=a_d (x-\alpha_1)\cdot\ldots\cdot(x-\alpha_d)\in\mathbb{Z}[x]$ be irreducible, then
\[
h(\alpha_1)=\frac{1}{d}\log\left( \vert a_d \vert \cdot \prod_{i=1}^d \max\{1,\vert \alpha_i\vert \}\right).
\]
For all upcoming information on height functions, we refer to \cite{BG}. This function $h$ satisfies $h(\alpha^n)=n h(\alpha)$ for all $\alpha \in \overline{\mathbb{Q}}^*$ and all $n\in\mathbb{N}=\{1,2,3,\ldots\}$, and vanishes precisely at $0$ and roots of unity. This makes it very easy to construct algebraic numbers of arbitrarily small positive height, since for any $\alpha \in \overline{\mathbb{Q}}^*$ which is not a root of unity, we have  
\begin{equation}\label{eq:tendtozero}
0 < h(\alpha^{\nicefrac{1}{n}})=\frac{1}{n}h(\alpha) \longrightarrow 0,
\end{equation}
as $n$ tends to infinity.

Let us denote the set of roots of unity by $\mu$. The famous Lehmer Conjecture, which has its origin in \cite{Le33}, postulates the existence of a constant $c>0$ such that $h(\alpha)\geq \frac{c}{[\mathbb{Q}(\alpha):\mathbb{Q}]}$ for all $\alpha\in\overline{\mathbb{Q}}^*\setminus \mu$. Consider the field
\[
\mathbb{Q}(\mu,\alpha, \alpha^{\nicefrac{1}{2}}, \alpha^{\nicefrac{1}{3}},\ldots)
\]
generated by all $n$th roots, $n=1,2,3,\ldots$, of an $\alpha \in \overline{\mathbb{Q}}^*$. This is, the field is generated by algebraic numbers, whose height is small for obvious reasons. But does this field also contain non-obvious algebraic numbers of small height; i.e. elements of small height not in $\mu\cup \{\alpha^{q} \vert q\in\mathbb{Q}\}$? A generalized version of the Lehmer conjecture, due to Ga\"el R\'emond, predicts a negative answer to this question, as we will explain in a moment.

 

Although our main result only deals with the multiplicative group $\overline{\mathbb{Q}}^*$, we will present this conjecture in greater generality. This is due to the fact that some parts of the paper are valid in the following general setting. 

Let $\mcG =A\times \mathbb{G}_m^N$ for some $N\in\mathbb{N}_0$ and an abelian variety $A$ defined over a number field $K$ equipped with an ample symmetric line bundle $\mathcal{L}$. The choice of this line bundle defines a N\'eron-Tate height $h_{\mathcal{L}}$ on $A(\overline{\mathbb{Q}})$. The canonical height $\widehat{h}_{\mcG}$ on $\mcG(\overline{\mathbb{Q}})$ is given as the sum of the N\'eron-Tate height and the Weil-height on each component. This means, for $(P,\alpha_1,\ldots,\alpha_N)\in\mcG(\overline{\mathbb{Q}})$ we set
\[
\widehat{h}_{\mcG}(P,\alpha_1,\ldots,\alpha_N)=h_{\mathcal{L}}(P)+\sum_{i=1}^N h(\alpha_i).
\]

If $G$ is any divisible group with a subgroup $\Gamma$, then we define the \emph{divisible hull} of $\Gamma$ to be the group 
\[
\Gamma_{\divi}:=\{\gamma \in G \vert n\gamma \in \Gamma \text{ for some } n \in \mathbb{N}\}.
\]
If $\Gamma$ is a subgroup of $\mcG(\overline{\mathbb{Q}})$, then we denote by $\End(\mcG)\cdot \Gamma$ the subgroup of $\mcG(\overline{\mathbb{Q}})$ generated by all elements of the form $\varphi(\gamma)$ with $\varphi\in\End(\mcG)=\End_{\overline{\mathbb{Q}}}(\mcG)$ and $\gamma\in\Gamma$.
Moreover we define
\[
\Gamma_{\sat}:= (\End(\mcG)\cdot \Gamma)_{\divi}.
\] 
Note that $\End(\mcG)\cdot \Gamma = \Gamma$ if $\End(\mcG)=\mathbb{Z}$. In this case there is no difference between $\Gamma_{\divi}$ and $\Gamma_{\sat}$. 
For any abelian group $G$ the \emph{rank} of $G$ is given by the maximal number of linearly independent elements in $\Gamma$. This rank will be denoted with $\rk(\Gamma)$. It follows that if $\Gamma$ as above has rank zero then $\Gamma_{\sat}=\Gamma_{\divi}$ is precisely given by the torsion subgroup $\mcG_{\tors}$. 

For any subgroup $\Gamma \subseteq \mcG(\overline{\mathbb{Q}})$ an element $\alpha \in \mcG(\overline{\mathbb{Q}})$ is called $\Gamma$-transversal, if it is contained in some translate $\gamma + B$, where $\gamma\in\Gamma_{\sat}$ and $B$ is a connected proper algebraic subgroup of $\mcG$. Now we can finally formulate R\'emond's generalized Lehmer conjecture \cite[Conjecture 3.4]{Re11}.

\begin{conjecture}[R\'emond]\label{conj}
Let $\mcG$ be either a torus or an abelian variety defined over a number field $K$, and let $\Gamma$ be a finite rank subgroup of $\mcG(\overline{\mathbb{Q}})$.
\begin{itemize}
\item[(a)] There exists a positive constant $c$ such that 
\[
\widehat{h}_{\mcG}(\alpha)\geq \frac{c}{[K(\Gamma_{\sat})(\alpha):K(\Gamma_{\sat})]^{\nicefrac{1}{\dim(\mcG)}}} \quad \forall ~  \alpha \in \mcG(\overline{\mathbb{Q}}) \text{ which are not } \Gamma\text{-transversal}.
\]
\item[(b)] For any $\varepsilon>0$ there is a positive constant $c_{\varepsilon}$ such that 
\[
\widehat{h}_{\mcG}(\alpha)\geq \frac{c_{\varepsilon}}{[K(\Gamma_{\sat})(\alpha):K(\Gamma_{\sat})]^{\nicefrac{1}{\dim(\mcG)}+\varepsilon}} \quad \forall ~  \alpha \in \mcG(\overline{\mathbb{Q}}) \text{ which are not } \Gamma\text{-transversal}.
\]
\item[(c)] For all finite extensions $L/K(\Gamma_{\sat})$ there is a positive constant $c_L$ such that 
\[
\widehat{h}_{\mcG}(\alpha)\geq c_L \quad \forall ~ \alpha \in \mcG(L)\setminus \Gamma_{\sat}.
\]
\end{itemize}
\end{conjecture}

Conjecture \ref{conj} is weaker than R\'emond's original conjecture in two ways. Firstly, \cite[Conjecture 3.4]{Re11} predicts lower bounds for the height of subvarieties of $\mcG$, not just for the height of points. Secondly, the exponent on the right hand side of (a) and (b) is smaller than our exponents $\frac{1}{\dim(\mcG)}$, resp. $\frac{1}{\dim(\mcG)}+\varepsilon$. Since the focus of this paper lies solely on part (c), we did not introduce the necessary notation to present the exponents conjectured by R\'emond. On the other hand, this conjecture could be generalized to cover also semi-abelian varieties of the form $A\times \mathbb{G}_m^N$. For part (c) this generalization of the conjecture can be found in \cite[Conjecture 1.2]{Pl19}.

Obviously, part (a) of Conjecture \ref{conj} implies part (b). It is also true that part (b) implies part (c). If $\dim(\mcG)=1$ this follows since in this case $\alpha$ is $\Gamma$-transversal if and only if $\alpha \in \Gamma_{\sat}$. For general $\mcG$ the implication (b) $\Rightarrow$ (c) follows as a very special case from the strong result \cite[Theorem 3.7]{Re11}.


There are some results on this conjecture in the case that the rank of $\Gamma$ is zero. Delsinne \cite{Del09} proved Conjecture \ref{conj} (b) in the case that $\mcG=\mathbb{G}^N$ and $\Gamma_{\sat}=\mcG_{\tors}$. Also under the assumption $\Gamma_{\sat}=\mcG_{\tors}$, Carrizosa \cite{Ca09} proved Conjecture \ref{conj} (b) in the case that $\mcG$ is an abelian variety with complex multiplication. (Previously, Delsinne's result for $N=1$ has been proven in \cite{AZ00}, and part (c) for $\mcG$ an abelian variety with complex multiplication and $\Gamma_{\sat}=\mcG_{\tors}$ has been proven in \cite{BS04}).  

\smallskip

Amoroso \cite{Am14} could achieve the following result towards the seemingly most easiest case of a group of positive rank: Let $\Gamma=\langle 2 \rangle$ be the subgroup of $\mcG=\mathbb{G}_m=\overline{\mathbb{Q}}^*$ generated by $2$, and define $\frac{1}{3^{\infty}}\Gamma=\{\alpha \in \overline{\mathbb{Q}}^* \vert \alpha^{3^n} \in \Gamma\}$. Then there is an effectively computable constant $c>0$ such that $h(\alpha)\geq c$ for all $\alpha \in \mathbb{Q}(\frac{1}{3^{\infty}}\Gamma)^* \setminus \frac{1}{3^{\infty}}\Gamma$. Under certain technical restrictions, a similar result for groups $\mcG=A\times \mathbb{G}_m$ and $\Gamma=\{0\} \times \langle b \rangle$, where $A$ is an elliptic curve and $b$ is an integer, has been announced in \cite{Pl19}.

\smallskip

We will give some weak group theoretic support for the validity of Conjecture \ref{conj} (c). The connection to group theory is given by the following observation, (cf. Lemma \ref{freemodgamma} below).

\begin{proposition}
If Conjecture \ref{conj} (c) is true, then $\nicefrac{\mcG(L)}{\Gamma_{\rm sat}}$ is free abelian, where $L$ is a finite extension of $K(\Gamma_{\rm sat})$.
\end{proposition}

Hence, the following theorem shows that a necessary condition for the truth of Conjecture \ref{conj} (c) is fulfilled. 

\begin{theorem}\label{thm:freeab}
Let $\mcG = \mathbb{G}_m$, and let $\Gamma$ be a subgroup of $\mcG(\overline{\mathbb{Q}})$ of finite rank. Then the group $\nicefrac{\mcG(L)}{\Gamma_{\sat}}$ is free abelian for all finite extensions $L/\mathbb{Q}(\Gamma_{\sat})$.
\end{theorem}

If the rank of $\Gamma$ is zero (i.e. if $\Gamma_{\sat}=\mcG_{\tors}$), then the statement of Theorem \ref{thm:freeab} is true for all semi-abelian varieties of the form $A\times \mathbb{G}_m^N$. This was proved by Bays, Hart and Pillay in the appendix of \cite{BHP}. We use their result to prove that the group $\nicefrac{\mcG(L)}{\Gamma_{\sat}}$ is free abelian if the torsion group of
\begin{equation}\label{eq:gr}
\nicefrac{\mcG(L)}{\mcG(K(\mcG_{\tors}))+\Gamma_{\sat}}
\end{equation}
has finite exponent. This reduction step is true in the general case $\mcG=A\times \mathbb{G}_m^N$, and follows from a classification result for free abelian groups due to Pontryagin.

In the last section we reduce to the case $\mcG=\mathbb{G}_m$ and prove Theorem \ref{thm:freeab}, by showing that the group from \eqref{eq:gr} is actually torsion free. This follows quite elementary by basic facts on cyclic field extensions.

%

\section{Some group theory and R\'emond's Lemma}

All abelian groups in this section will be written additively. 

\begin{definition}\label{def:norm}
A \emph{norm} on an abelian group $G$ is a function $\norm{\cdot}: G \longrightarrow \mathbb{R}_{\geq 0}$ satisfying
\begin{enumerate}[(i)]
\item $\norm{g} = 0 ~ \Longleftrightarrow ~ g = 0$ is the neutral element,
\item $\norm{ g + f} \leq \norm{ g } + \norm{ f }$ for all $g,f \in G$, and
\item $\norm{ n\cdot g } = \vert n \vert \cdot \norm{ g }$ for all $g\in G$ and all $n \in \mathbb{Z}$.
\end{enumerate}
If $\norm{\cdot}$ only satisfies (ii) and (iii), then it is called a \emph{semi-norm}. A norm $\norm{\cdot}$ is called \emph{discrete} on $G$ if and only if $0$ is not an accumulation point in the set $\{\norm{g} \vert g\in G\}$. 
\end{definition}

In this section $\mcG$ always refers to a semi-abelian variety of the form $A\times \mathbb{G}_m^N$, defined over a number field $K$, where $A$ is an abelian variety and $N\in\mathbb{N}_0$. One of the main properties of $\widehat{h}_{\mcG}$ is that it is well-defined on $\nicefrac{\mcG(\overline{\mathbb{Q}})}{\mcG_{\tors}}$, where $\mcG_{\tors}$ is the torsion subgroup of $\mcG$. Moreover, the map
\[
(P,\alpha_1,\ldots,\alpha_N) \mapsto \sqrt{h_{\mathcal{L}}(P)}+\sum_{i=1}^N h(\alpha_i)
\]
is a norm on $\nicefrac{\mcG(\overline{\mathbb{Q}})}{\mcG_{\tors}} = \nicefrac{A(\overline{\mathbb{Q}})\times\mathbb{G}_m^N(\overline{\mathbb{Q}})}{\mcG_{\tors}}$. We will denote this norm on $\nicefrac{\mcG(\overline{\mathbb{Q}})}{\mcG_{\tors}}$ by $\Vert \cdot \Vert_h$. 

Let $\norm{.}$ be a norm on a divisible group $G$ and let $\Gamma$ be a subgroup of $G$. Then we define the function 
\[
\norm{.}_{\Gamma}: \nicefrac{G}{\Gamma_{\divi}} \longrightarrow \mathbb{R}_{\geq0} \quad ; \quad \norm{[\alpha]}_{\Gamma}= \inf\{\norm{\alpha+\gamma} \vert \gamma \in \Gamma_{\divi}\} .
\]
Since this map is well-defined we will simply write $\norm{\alpha}_{\Gamma}$ for $\norm{[\alpha]}_{\Gamma}$.

The following lemma is along the lines of \cite[Lemma 3.5]{Re11}. The proof is also due to Ga\"el R\'emond, who presented a special case of it at the workshop on ``Heights in Diophantine geometry, group theory and additive combinatorics'' at the ESI in Vienna in November 2013. See also \cite{Gr17} for a quantitative version in the $G=\mathbb{G}_m$ case.

\begin{lemma}\label{Bogomolov}
Let $G$ be a divisible group such that there is a norm $\norm{.}$ on $\nicefrac{G}{G_{\tors}}$. Let $\Gamma$ and $H$ be subgroups of $G$ such that
\begin{itemize}
\item $\rk(\Gamma)=\rk(\Gamma \cap H)$ is finite, and
\item $\norm{g}\geq \kappa$ for all $g\in H\setminus \Gamma_{\divi}$, for some constant $\kappa >0$.
\end{itemize}
Then there exists a positive constant $c$ only depending on $H$ and $\Gamma$ such that $\norm{g}_{\Gamma} \geq c$ for all $g\in H\setminus \Gamma_{\divi}$.
\end{lemma}

\begin{proof} 
We set $\rk(\Gamma)=r$ and choose linearly independent elements $\gamma_1,\ldots,\gamma_r\in\Gamma\cap H$. 

Let $g \in H\setminus \Gamma_{\divi}$, and let $\underline{b}=(b_1,\dots,b_r)$ be an element in $\mathbb{Q}^r$. We aim to bound 
\[
\norm{g+b_1\gamma_1+\ldots+b_r\gamma_r}
\]
from below by a constant independent on $g$. To this end, let $Q\geq 2r\max_{1\leq i \leq r} \norm{\gamma_i} \kappa^{-1}$ be an integer. By Dirichlet's approximation theorem, there is a positive integer $m$ and an $\underline{a}\in\frac{1}{m}\mathbb{Z}^r$ such that
\begin{equation}\label{fin}
 m\leq Q^r \text{ and } \vert a_i - b_i \vert \leq \frac{1}{mQ} \text{ for all } i \in \{1,\ldots,r\}.
\end{equation}
Since $g\in H\setminus\Gamma_{\divi}$ and $\gamma_1,\ldots,\gamma_r \in H$, we have $mg+ma_1\gamma_1+\ldots+ma_r\gamma_r \in H\setminus\Gamma_{\divi}$. Hence, by our assumptions on $H$ we have
\begin{equation}\label{bound}
\norm{g+a_1\gamma_1+\ldots+a_r\gamma_r} = \frac{1}{m}\norm{mg+ma_1\gamma_1+\ldots+ma_r\gamma_r} \geq \frac{\kappa}{m}.
\end{equation}
Moreover, applying properties (ii) and (iii) of Definition \ref{def:norm} for $\norm{.}$, yields
 \begin{align}\label{lipschitz} \big\vert \norm{g+a_1 \gamma_1+ \dots + a_r\gamma_r} - \norm{g + b_1 \gamma_1 + \dots + b_r \gamma_r} \big\vert &\leq \norm{(a_1 - b_1) \gamma_1 + \dots + (a_r - b_r)\gamma_r} \nonumber \\ &\leq \max_{1\leq i \leq r} \{\norm{\gamma_i}\} \cdot \sum_{i=1}^r \vert a_i - b_i \vert.
\end{align}
Combining \eqref{lipschitz} and \eqref{bound} yields that $\norm{g + b_1 \gamma_1 + \dots + b_r \gamma_r}$ is bounded from below by
\begin{align*}
 &\norm{g + a_1 \gamma_1+ \dots + a_r \gamma_r} - \big\vert \norm{g+a_1 \gamma_1+ \dots + a_r\gamma_r} - \norm{g + b_1 \gamma_1 + \dots + b_r \gamma_r} \big\vert \\
\geq & \frac{\kappa}{m} - \max_{1\leq i \leq r} \{\norm{\gamma_i}\} \cdot \sum_{i=1}^r \vert a_i - b_i \vert\overset{\eqref{fin}}{\geq} \frac{\kappa}{m} - \frac{\max_{1\leq i \leq r} \{\norm{\gamma_i}\} \cdot r}{mQ}\\  = & \frac{\kappa Q - \max_{1\leq i \leq r} \{\norm{\gamma_i}\} \cdot r}{mQ} \geq \frac{\max_{1\leq i \leq r} \{\norm{\gamma_i}\} \cdot r}{Q^{r+1}}.
\end{align*}
The latter is the postulated positive constant $c$ which only depends on $H$ and $\Gamma$.
\end{proof}

Let $\Gamma$ be a subgroup of $\mcG(\overline{\mathbb{Q}})$. Since $\End(\mcG)$ is finitely generated as an additive group, the group $\Gamma_{\sat}$ is of finite rank, whenever the rank of $\Gamma$ is finite. Hence, the next result is an immediate consequence of Lemma \ref{Bogomolov} if we set $G=\mcG(\overline{\mathbb{Q}})$, and $H=\mcG(F)$ for some field $F\subseteq\overline{\mathbb{Q}}$.

\begin{corollary}\label{cor:mcG}
Let $\Gamma \subseteq \mcG(\overline{\mathbb{Q}})$ be a subgroup of finite rank, and let $F$ be a subfield of $\overline{\mathbb{Q}}$ satisfying
\begin{enumerate}[(i)]
\item $\rk(\Gamma_{\sat})=\rk(\Gamma_{\sat}\cap \mcG(F))$, and
\item there is a positive constant $\kappa$ such that $\widehat{h}_{\mcG}(\alpha) \geq \kappa$ for all $\alpha \in \mcG(F)\setminus \Gamma_{\sat}$.
\end{enumerate}
Then there is a positive constant $c$ only depending on $F$ and $\Gamma$ such that $\norm{\alpha}_{h,\Gamma}\geq c$ for all $\alpha \in \mcG(F)\setminus \Gamma_{\sat}$.
\end{corollary}

\begin{lemma}\label{heightisnorm}
Let $\Gamma$ be a subgroup of $\mcG(\overline{\mathbb{Q}})$ and $\norm{.}$ a norm on $\nicefrac{\mcG(\overline{\mathbb{Q}})}{\mcG_{\tors}}$. 
Then the function $\norm{.}_{\Gamma}$ is a seminorm on $\nicefrac{\mcG(\overline{\mathbb{Q}})}{\Gamma_{\sat}}$. If $\Gamma$ has finite rank, then the particular function $\norm{.}_{h,\Gamma}$ is a norm.
\end{lemma} 

\begin{proof} First we will show that $\norm{.}_{\Gamma}$ is a semi-norm, without additional assumptions on the group $\Gamma$. In order to do so, we have to check the properties (ii) and (iii) from Definition \ref{def:norm}. 
These properties follow from the respective properties of the norm $\norm{.}$. 

The triangular inequality follows from
\begin{align*}
 \norm{\alpha}_{\Gamma}+\norm{\beta}_{\Gamma} &=\inf\{\norm{\alpha+\gamma} \vert \gamma \in \Gamma_{\divi}\}+\inf\{\norm{\beta+\gamma'} \vert \gamma' \in \Gamma_{\divi}\}\\
 &=\inf\{\norm{\alpha+\gamma} + \norm{\beta+\gamma'} \vert \gamma,\gamma' \in \Gamma_{\divi}\}\\
 &\geq \inf\{\norm{\alpha+\beta+\gamma+\gamma'} \vert \gamma,\gamma' \in \Gamma_{\divi}\} \\
 &=\inf\{\norm{\alpha+\beta+\gamma''} \vert \gamma'' \in \Gamma_{\divi}\} = \norm{\alpha+\beta}_{\Gamma},
\end{align*}
and the homogeneity follows similarly from the equation
\begin{align*}
 \norm{n\alpha}_{\Gamma} &=\inf\{\norm{n\alpha + \gamma} \vert \gamma \in \Gamma_{\divi}\} =\inf\{\norm{n(\alpha+\gamma')} \vert \gamma' \in \Gamma_{\divi}\}\\
 &=\inf\{n\norm{\alpha+\gamma'} \vert \gamma' \in \Gamma_{\divi}\} =n\norm{\alpha}_{\Gamma}.
\end{align*}
From now on we assume that $\Gamma$ is of finite rank and that the norm $\norm{.}=\norm{.}_h$ is induced by the canonical height $\widehat{h}_{\mcG}$ on $\mcG$. Then also $\Gamma_{\sat}$ is of finite rank $r<\infty$. Let $\gamma_1,\dots,\gamma_r$ be linearly independent elements in $\Gamma_{\sat}$. We are left to prove property (i) from Definition \ref{def:norm}. Obviously we have $\norm{\alpha}_{h,\Gamma}=0$ for all $\alpha \in \Gamma_{\sat}$. By Northcott's theorem, $\norm{.}_{h}$ is discrete on $\nicefrac{\mcG(F)}{\mcG_{\tors}(F)}$ for all number fields $F$.
Hence, if $\alpha \in \mcG(\overline{\mathbb{Q}})\setminus \Gamma_{\sat}$ is arbitrary we set $F=\mathbb{Q}(\gamma_1,\dots,\gamma_r,\alpha)$ and apply Corollary \ref{cor:mcG}. This yields $\norm{\alpha}_{h,\Gamma} \neq 0$ and concludes the proof.
\end{proof}

We use the notation from Conjecture \ref{conj}. After replacing $\Gamma$ by the finite rank subgroup $\Gamma_{\sat}$ of $\mcG(\overline{\mathbb{Q}})$, Lemma \ref{heightisnorm} tells us that $\norm{\cdot}_{h,\Gamma_{\sat}}$ is a norm on $\nicefrac{\mcG(L)}{\Gamma_{\sat}}$. This norm is discrete if and only if the statement of Conjecture \ref{conj} (c) is true. Therefore, Conjecture \ref{conj} (c) is true if and only if $\norm{\cdot}_{h,\Gamma_{\sat}}$ is a discrete norm on $\nicefrac{\mcG(L)}{\Gamma_{\sat}}$.

Note that there exists a discrete norm on an abelian group if and only if this group is free abelian. This result was proved independently by Lawrence \cite{La84} and Zorzitto \cite{Zo85} for countable groups, and by Stepr\=ans \cite{St85} in the general case. As this result is the bridge between Conjecture \ref{conj} and Theorem \ref{thm:freeab}, we state it as a proposition. 

\begin{proposition}\label{LSZ}
An abelian group $G$ is free if and only if there is a discrete norm on $G$.
\end{proposition}

We conclude

\begin{lemma}
\begin{align*}
\text{Conjecture \ref{conj} (c) is true } &\Longleftrightarrow ~ \norm{\cdot}_{h,\Gamma_{\sat}} \text{ is a discrete norm on } \nicefrac{\mcG(L)}{\Gamma_{\sat}}\\
 &\Longrightarrow ~ \nicefrac{\mcG(L)}{\Gamma_{\sat}} \text{ is free abelian}
\end{align*}
\end{lemma}

Hence, Theorem \ref{thm:freeab} tells us, that at least there exists \emph{some} discrete norm on $\nicefrac{\mcG(L)}{\Gamma_{\sat}}$ if $\mcG=\mathbb{G}_m$.

\section{Criterion for the freeness of an abelian group}

Again we fix some $\mcG = A\times \mathbb{G}_m^N$ defined over a number field $K$, where $A$ is an abelian variety and $N \in \mathbb{N}_0$. 

\begin{lemma}\label{freemodgamma}
Let $F$ be any subfield of $K(\Gamma_{\sat})$. The group $\nicefrac{\mcG(K(\Gamma_{\sat}))}{\Gamma_{\sat}}$ is free abelian if for every field $E$, with $F\subseteq E \subseteq K(\Gamma_{\sat})$ and $[E:F]< \infty$, we have 
\begin{enumerate}[(i)]
\item $\nicefrac{\mcG(E)}{(\Gamma_{\sat}\cap \mcG(E))}$ is free abelian, and
\item the torsion group of $\nicefrac{\mcG(K(\Gamma_{\sat}))}{\mcG(E) + \Gamma_{\sat}}$ has finite exponent.
\end{enumerate}
\end{lemma}

This is mainly an application of a classification result of Pontryagin. The proof follows very closely the proofs of \cite[Lemma 1]{Ma72} and \cite[Proposition 2.3]{GHP}.

\begin{proof} 
By a theorem of Pontryagin (cf. \cite[Theorem VI.2.3]{EM}) an abelian group $G$ is free abelian, if every finite subset of $G$ is contained in a free abelian subgroup $H\subseteq G$ such that $\nicefrac{G}{H}$ is torsion free.

Therefore, let $S=\{[\alpha_1],\ldots,[\alpha_s]\}\subseteq \nicefrac{\mcG(K(\Gamma_{\sat}))}{\Gamma_{\sat}}$ and set $E:=K(\mcG_{\tors},\alpha_1,\ldots,\alpha_s)$. Obviously $E$ is a finite extension of $K(\mcG_{\tors})$ and we have 
\[
S \subseteq \nicefrac{\mcG(E)+\Gamma_{\sat}}{\Gamma_{\sat}}\cong \nicefrac{\mcG(E)}{\mcG(E)\cap \Gamma_{\sat}}.
\]
Let $m\in\mathbb{N}$ be the exponent of the torsion subgroup of
\[
\nicefrac{\left( \nicefrac{\mcG(K(\Gamma_{\sat}))}{\Gamma_{\sat}} \right)}{\left( \nicefrac{\mcG(E)+\Gamma_{\sat}}{\Gamma_{\sat}}  \right)}\cong \nicefrac{\mcG(K(\Gamma_{\sat}))}{\mcG(E)+\Gamma_{\sat}}.
\]
Note, that this exponent is indeed an element of $\mathbb{N}$, by assumption (ii) of the lemma. Now define
\[
H:=\left\{[\alpha]\in \nicefrac{\mcG(K(\Gamma_{\sat}))}{\Gamma_{\sat}} \vert m\cdot [\alpha] \in \nicefrac{\mcG(E)+\Gamma_{\sat}}{\Gamma_{\sat}} \right\}.
\]
By assumption (i) the group $\nicefrac{\mcG(E)+\Gamma_{\sat}}{\Gamma_{\sat}}\cong \nicefrac{\mcG(E)}{\mcG(E)\cap \Gamma_{\sat}}$ is free abelian, and hence $H$ is free abelian. Moreover, by construction the quotient of $\nicefrac{\mcG(K(\Gamma_{\sat}))}{\Gamma_{\sat}}$ by $H$ is torsion free. Hence, $S \subseteq H$ and $H$ satisfies the hypothesis of Pontryagins theorem. It follows that under the assumptions (i) and (ii) the group $\nicefrac{\mcG(K(\Gamma_{\sat}))}{\Gamma_{\sat}}$ is free abelian.
\end{proof}

As stated in the introduction, we will use the result from Bays, Hart, and Pillay that $\nicefrac{\mcG(K(\mcG_{\tors}))}{\mcG_{\tors}}$ is free abelian. The case $\mcG=\mathbb{G}_m$ is originally due to Iwasawa \cite{Iw53}, and the case $\mcG=A$ is originally due to Larsen \cite{La05}. We will sketch a proof of their result.

\begin{theorem}[\cite{BHP}, Lemma A.7]\label{thm:BHP}
Let $\mcG$ and $K$ be as above and let $K'/K$ be finite. Then
\begin{enumerate}[(i)]
\item $\nicefrac{\mcG(K')}{\mcG(K')\cap\mcG_{\tors}}$ is free abelian, and
\item the exponent of the torsion group of $\nicefrac{\mcG(K'(\mcG_{\tors}))}{\mcG(K')+\mcG_{\tors}}$ is finite.
\end{enumerate}
In particular, $\nicefrac{\mcG(K(\mcG_{\tors}))}{\mcG_{\tors}}$ is a free abelian group.
\end{theorem}
\begin{proof}
In this proof we use the language of continuous group cohomology. The field $K'$ is a number field. Hence the norm $\norm{\cdot}_h$ induced by the canonical height of $\mcG$ is discrete on $\nicefrac{\mcG(K')}{\mcG(K')\cap\mcG_{\tors}}$. Hence, statement (i) follows from Proposition \ref{LSZ}. In order to prove (ii), let $\alpha \in \mcG(K'(\mcG_{\tors}))$ be a representative of a torsion point in $\nicefrac{\mcG(K'(\mcG_{\tors}))}{\mcG(K')+\mcG_{\tors}}$. Then, the map $\tau \mapsto \tau(\alpha)-\alpha$ represents an element in $H^1(\Gal(K'(\mcG_{\tors})/K'),\mcG[n])$ for some $n\in \mathbb{N}$. By Lemma A.3 from \cite{BHP}, there is a constant $c$ only depending on $\mcG$ and $K'$ such that $\tau \mapsto c\cdot(\tau(\alpha)-\alpha)$ is equivalent to the zero map in $H^1(\Gal(K'(\mcG_{\tors})/K'),\mcG[n])$. Hence, there is an element $P\in\mcG[n]$ such that
\[
c\cdot(\tau(\alpha)-\alpha)=\tau(P)-P \quad \forall ~ \tau \in \Gal(K'(\mcG_{\tors})/K').
\]
It follows
\[
\tau(c\cdot \alpha - P) = c\cdot \alpha - P \quad \forall ~ \tau \in \Gal(K'(\mcG_{\tors})/K'),
\]
and hence $c\cdot \alpha -P \in \mcG(K')$, respectively $c\cdot \alpha \in \mcG(K') + \mcG_{\tors}$. This means that the order of the residue class of $\alpha \in \nicefrac{\mcG(K'(\mcG_{\tors}))}{\mcG_{\tors}}$ divides $c$, which proves statement (i).

If we apply Lemma \ref{freemodgamma} with $\Gamma=\{0\}$ and $F=K$, it follows from (i) and (ii) that $\nicefrac{\mcG(K(\mcG_{\tors}))}{\mcG_{\tors}}$ is free abelian.
\end{proof}

\begin{proposition}\label{prop:firstclaim}
Let $K'/K$ be finite and let $\Gamma$ be a finite rank subgroup of $\mcG(\overline{\mathbb{Q}})$. For any field $K \subseteq E \subseteq K'(\mcG_{\tors})$, the group $\nicefrac{\mcG(E)}{\mcG(E)\cap \Gamma_{\sat}}$ is free abelian.
\end{proposition}
\begin{proof} 
By Theorem \ref{thm:BHP} we know that $\nicefrac{\mcG(K'(\mcG_{\tors}))}{\mcG_{\tors}}$ is free abelian. Therefore $\nicefrac{\mcG(E)}{\mcG(E)_{\tors}}$ is, as a subgroup, free abelian. Hence by Proposition \ref{LSZ} there is a discrete norm $\norm{.}$ on $\nicefrac{\mcG(E)}{\mcG(E)_{\tors}}$.

Set $\tilde{\Gamma}=\Gamma_{\sat}\cap \mcG(E)$ which is a finite rank subgroup, since $\Gamma_{\sat}$ is. The discrete norm $\norm{.}$ extends uniquely to a norm on $\nicefrac{\mcG(E)_{\divi}}{\mcG_{\tors}}$. Hence we can apply Lemma \ref{Bogomolov} to deduce the existence of a constant $c > 0$ such that 
\begin{equation}\label{eq:firstpart}
 \norm{\alpha}_{\tilde{\Gamma}} = \inf \{\norm{\alpha+\gamma} \vert \gamma\in \tilde{\Gamma}_{\divi} \} \geq c \text{ for all } \alpha \in \mcG(E)\setminus \tilde{\Gamma}_{\divi}.
\end{equation}
By Lemma \ref{heightisnorm} we already know that $\norm{.}_{\tilde{\Gamma}}$ is a seminorm on $\nicefrac{\mcG(E)}{\tilde{\Gamma}}$. Therefore \eqref{eq:firstpart} tells us that $\norm{.}_{\tilde{\Gamma}}$ is actually a discrete norm. If we apply Proposition \ref{LSZ} once more, we achieve that $\nicefrac{\mcG(E)}{\tilde{\Gamma}}$ is free abelian.
\end{proof}

\section{Proof of Theorem \ref{thm:freeab}}

From now on we fix $\mcG=\mathbb{G}_m$; i.e. we work in the multiplicative group of $\overline{\mathbb{Q}}^*$. Recall that we denote by $\mu$ the set of all roots of unity.

\begin{proposition}\label{prop:gmtorsionfree}
Let $\Gamma=\langle \gamma_1,\ldots,\gamma_r \rangle$ be a subgroup of $\overline{\mathbb{Q}}^*$, and let $K$ be a number field. We set $L=K(\Gamma_{\sat})$ and let $E\subseteq L$ be a finite extension of $K(\mu,\gamma_1,\ldots,\gamma_r)$. Then the group $\nicefrac{L^*}{E^* \Gamma_{\sat}}$ is torsion free.
\end{proposition}
\begin{proof}
Let $[\alpha]\in \nicefrac{L^*}{E^* \Gamma_{\sat}}$ be a torsion point. Then there exists a natural number $n$ with $\alpha^n \in E$.

Since $\alpha$ is in $L=K(\Gamma_{\sat})$, we have $E(\alpha)\subseteq E(\gamma_1^{\nicefrac{1}{m_1}},\dots,\gamma_r^{\nicefrac{1}{m_r}})$ for some $m_1,\dots,m_r \in \mathbb{N}$. We set $E_0 = E$ and define for every $i\in\{1,\dots,r\}$ the field
\[
E_i = E(\gamma_1^{\nicefrac{1}{m_1}},\dots,\gamma_i^{\nicefrac{1}{m_i}}).
\]
Every extension $E_i / E_{i-1}$ in the chain
\begin{equation}
E=E_0 \subseteq E_1 \subseteq E_2 \subseteq \cdots \subseteq E_r
\end{equation}
is cyclic of some order $k_i \mid m_i$. Note, that $\gamma_i \in E$, for all $i\in \{1,\dots,r\}$. 
By the classic theory of cyclic extensions (cf. \cite{La02}, Chapter VI), every intermediate field of $E_{i}=E_{i-1}(\gamma_i^{\nicefrac{1}{m_i}})/E_{i-1}$ is given by $E_{i-1}(\gamma_i^{\nicefrac{d}{m_i}})$ for some $d\in\mathbb{N}$. Hence, $E_{r-1}(\alpha)=E_{r-1}(\gamma_r^{\nicefrac{d_r}{m_r}})$. Assume, that the degree of $\alpha$ over $E_{r-1}$ is $n_r$ and that $\sigma$ is a generator of $\Gal(E_{r-1}(\alpha)/E_{r-1})$. Then, since $\alpha^n \in E \subseteq E_{r-1}$, we have 
\[
\sigma(\gamma_r^{\nicefrac{d_r}{m_r}})=\zeta_{n_r}\gamma_r^{\nicefrac{d_r}{m_r}} \quad \text{ and } \quad \sigma(\alpha)=\zeta_{n_r}^{l_r} \alpha,
\]
where $\zeta_{n_r}$ is a primitive $n_r$-th root of unity and $l_r \in \mathbb{N}$. We can conclude that $\sigma$, and hence $\Gal(E_{r-1}(\alpha)/E_{r-1})$, acts trivial on the element $\nicefrac{\alpha}{\gamma_r^{\nicefrac{d_r l_r}{m_r}}}$. It follows 
\[
\frac{\alpha}{\gamma_r^{\nicefrac{d_r l_r}{m_r}}} \in E_{r-1} \quad \text{ and } \quad \left(\frac{\alpha}{\gamma_r^{\nicefrac{d_r l_r}{m_r}}}\right)^{n m_r} \in E.
\]
Thus we can repeat this argument with $n$ replaced by $n m_r$, and $\alpha$ replaced by $\nicefrac{\alpha}{\gamma_r^{\nicefrac{d_r l_r}{m_r}}}$. Induction yields
\[
\frac{\alpha}{\prod_{i=1}^{r}\gamma_i^{\nicefrac{d_i l_i}{m_i}}} \in E_0^*
\]
which is equivalent to $\alpha \in E^* \Gamma_{\sat}$. Hence the residue class of $\alpha$ in the group $\nicefrac{L^*}{E^* \Gamma_{\sat}}$ is trivial, meaning that the group is torsion free. 
\end{proof}

\begin{proof}[Proof of Theorem \ref{thm:freeab}]
Let $L$ be a finite extension of $K(\Gamma_{\sat})$; say $L=K'(\Gamma_{\sat})$ for a finite extension $K'/K$. Set $F=K'(\gamma_{1},\ldots,\gamma_{r},\mu)$ and let $E/F$ be any finite extension such that $E\subseteq L$. By Proposition \ref{prop:firstclaim} we know that $\nicefrac{E^*}{E^*\cap \Gamma_{\divi}}$ is free abelian and we have just seen that $\nicefrac{L^*}{E^* \Gamma_{\divi}}$ is torsion free. Hence, the assumptions from Lemma \ref{freemodgamma} are met, which proves Theorem \ref{thm:freeab}.
\end{proof}

We conclude this paper by the following remark.

\begin{remark}
If we allow $\Gamma$ in Conjecture \ref{conj} (c) to be of infinite rank, then anything can happen. We still assume $\mcG=\mathbb{G}_m$. On the one hand, the statement of Conjecture \ref{conj} (c) is vacuously true if $\Gamma = \overline{\mathbb{Q}}^*$. On the other hand, if $\Gamma=\Gamma_{\sat}$ is the group of all algebraic units, then $\mathbb{Q}(\Gamma)=\overline{\mathbb{Q}}$, but obviously there are algebraic numbers of small height which are not in $\Gamma$.

A slightly less trivial example of a group $\Gamma$ of infinite degree that satisfies Conjecture \ref{conj} (c) is the following. Take $\Gamma=\Gamma_{\sat}$ as the group of all algebraic numbers such that all Galois conjugates lie on the unit circle. Then we have $\mathbb{Q}(\Gamma)=\mathbb{Q}^{\rm tr}(i)$, where $\mathbb{Q}^{tr}$ denotes the maximal totally real extension of $\mathbb{Q}$ (see for instance \cite[Theorem 1]{BL78}). Now Schinzel's result \cite[Theorem 2]{Sch73} implies that $h(\alpha)\geq \frac{1}{2}\log\left(\frac{1+\sqrt{5}}{2}\right)$ for all $\alpha \in \mathbb{Q}(\Gamma_{\sat})^*\setminus \Gamma_{\sat}$.
\end{remark}

\bigskip

{\it Acknowledgements:} I would like to thank Ga\"el R\'emond for providing the proof of Lemma \ref{Bogomolov}, which initiated this project, and for invaluable comments on an early version of this manuscript. Moreover, special thanks go to Arno Fehm for many interesting discussions on this topic.

\end{document}